\documentclass[11pt]{article}
\usepackage{amsmath,amssymb,amsthm}

\newtheorem{theorem}{Theorem}[section]
\newtheorem{lemma}[theorem]{Lemma}
\newtheorem{proposition}[theorem]{Proposition}
\newtheorem{corollary}[theorem]{Corollary}

\theoremstyle{definition}

\numberwithin{equation}{section}

\renewcommand{\phi}{\varphi}

\newcommand{\Aut}{\operatorname{Aut}}
\newcommand{\Ker}{\operatorname{Ker}}

\newcommand{\N}{\mathbb{N}}
\newcommand{\Z}{\mathbb{Z}}
\newcommand{\R}{\mathbb{R}}
\newcommand{\T}{\mathbb{T}}

\title{Some remarks on topological full groups \\
of Cantor minimal systems II}
\author{Hiroki Matui \\
Graduate School of Science \\
Chiba University \\
Inage-ku, Chiba 263-8522, Japan}
\date{}

\begin{document}
\maketitle

\begin{abstract}
We prove that commutator subgroups of topological full groups 
arising from minimal subshifts have exponential growth. 
We also prove that the measurable full group associated to 
the countable, measure-preserving, ergodic and hyperfinite 
equivalence relation is topologically generated by two elements. 
\end{abstract}

\section{Introduction}

We study algebraic properties of topological full groups of 
Cantor minimal systems. 
By a Cantor set, we mean a metrizable topological space 
which is compact, totally disconnected 
(the closed and open sets form a base for the topology) and 
has no isolated points. 
Any two such spaces are homeomorphic. 
A homeomorphism $\phi:X\to X$ is said to be minimal 
if for all $x\in X$ the set $\{\phi^n(x)\mid n\in\Z\}$ is dense in $X$, 
or equivalently, 
there are no non-trivial closed $\phi$-invariant subsets of $X$. 
A pair $(X,\phi)$ of 
a Cantor set $X$ and a minimal homeomorphism $\phi$ of it 
is called a Cantor minimal system. 
The study of orbit structure of such dynamical systems was initiated 
by T. Giordano, I. F. Putnam and C. F. Skau in \cite{GPS95crelle}. 
They classified Cantor minimal systems 
up to topological orbit equivalence. 
This classification was later extended 
to cover all minimal actions of finitely generated abelian groups 
on Cantor sets \cite{GMPS}. 

For a Cantor minimal system $(X,\phi)$, 
the topological full group $[[\phi]]$ was introduced in \cite{GPS99Israel}. 
The group $[[\phi]]$ consists of all homeomorphisms $\psi:X\to X$ 
for which there exists a continuous map $c:X\to\Z$ 
such that $\psi(x)=\phi^{c(x)}(x)$. 
Clearly $[[\phi]]$ is infinite and countable. 
It was shown in \cite[Corollary 4.4]{GPS99Israel} that 
$[[\phi_1]]$ is isomorphic to $[[\phi_2]]$ as an abstract group 
if and only if $\phi_1$ is conjugate to $\phi_2$ or $\phi_2^{-1}$. 
This result suggests that 
the algebraic structure of the topological full group $[[\phi]]$ is 
rich enough to recover the dynamics of $\phi$. 
Since then 
various properties of $[[\phi]]$ have been studied in \cite{M06IJM,BM,GM}. 
We collect below some of them. 
The commutator subgroup $D([[\phi]])$ of $[[\phi]]$ is simple 
(\cite[Theorem 4.9]{M06IJM} and \cite[Theorem 3.4]{BM}). 
The quotient group $[[\phi]]/D([[\phi]])$ is isomorphic to 
the direct sum of $\Z$ and 
$C(X,\Z/2\Z)/\{f-f\circ\phi\mid f\in C(X,\Z/2\Z)\}$ 
(\cite[Section 5]{GPS99Israel} and \cite[Theorem 4.8]{M06IJM}), 
where $C(X,\Z/2\Z)$ denotes 
the $\Z/2\Z$-valued continuous functions on $X$ with pointwise addition. 
The commutator subgroup $D([[\phi]])$ is finitely generated 
if and only if $\phi$ is a minimal subshift over a finite alphabet 
(\cite[Theorem 5.4]{M06IJM}). 
Recently, R. Grigorchuk and K. Medynets \cite{GM} proved that 
$[[\phi]]$ is locally embeddable into finite groups. 
It is not yet known if $[[\phi]]$ is amenable. 

In the present paper, we prove a couple of new results about $[[\phi]]$. 
As mentioned above, for a minimal subshift $\phi$, 
$D([[\phi]])$ is finitely generated. 
It is then natural to consider the growth of $D([[\phi]])$. 
We first observe that when $\phi$ is an odometer, 
any finitely generated subgroup of $[[\phi]]$ has polynomial growth 
(Proposition \ref{odmtr}). 
Next, we prove that $D([[\phi]])$ contains the lamplighter group 
if and only if $\phi$ is not an odometer 
(Theorem \ref{lamplighter}). 
In particular, this implies that when $\phi$ is a minimal subshift, 
$D([[\phi]])$ has exponential growth (Corollary \ref{exponential}). 
Existence (or non-existence) of finitely generated subgroups 
of intermediate growth remains open. 
In Section 3, we discuss generators 
of the topological full group $[[\phi]]$ of Sturmian shifts $\phi$. 
We have already shown in \cite[Example 6.2]{M06IJM} that 
$[[\phi]]$ is generated by three elements. 
Based on this result, 
J. Kittrell and T. Tsankov \cite{KT10Ergodic} proved that 
the measurable full group associated to 
the countable, measure-preserving, ergodic and hyperfinite 
equivalence relation on the standard probability space is 
topologically generated by at most three elements. 
In this paper we shall show that 
$D([[\phi]])$ is contained in a subgroup generated by two elements 
(Proposition \ref{sturmian}). 
By using this, we can conclude that 
the measurable full group of the hyperfinite equivalence relation is 
topologically generated by two elements (Theorem \ref{t=2}). 
Also, this readily improves the estimates obtained in \cite{KT10Ergodic} 
for the number of topological generators of certain measurable full groups 
(Corollary \ref{cost}).

\section{Growth of topological full groups}

In this section, we discuss 
growth of (finitely generated subgroups of) topological full groups. 
The reader may consult \cite{Harpe} for basic theory of growth of groups. 

For $m\in\N$, 
we write $\Z_m=\Z/m\Z$ and identify it with $\{0,1,\dots,m{-}1\}$. 
The cardinality of a set $F$ is written $\lvert F\rvert$. 

Let us first recall the odometers. 
Let $(m_n)_n$ be a sequence of natural numbers such that 
$m_n$ divides $m_{n+1}$ and $m_n\to\infty$ as $n\to\infty$. 
There exists a surjective homomorphism $\rho_n:\Z_{m_{n+1}}\to\Z_{m_n}$ 
such that $\rho_n(1)=1$. 
We let $X$ be the inverse limit of $\Z_{m_n}$ under the map $\rho_n$, 
that is, 
\[
X=\left\{(x_n)_n\in\prod\Z_{m_n}\mid\rho_n(x_{n+1})=x_n\right\}. 
\]
With the product topology, $X$ is a Cantor set. 
Define a homeomorphism $\phi:X\to X$ by $\phi((x_n)_n)=(x_n+1)_n$. 
It is easy to see that $(X,\phi)$ is a Cantor minimal system. 
We call $(X,\phi)$ the odometer of type $(m_n)_n$. 

\begin{proposition}\label{odmtr}
Let $(X,\phi)$ be the odometer of type $(m_n)_n$. 
The topological full group $[[\phi]]$ is written as 
an increasing union of subgroups of the form $\Z^{m_n}\rtimes S_{m_n}$, 
where the symmetric group $S_{m_n}$ acts on $\Z^{m_n}$ 
by permutations of the coordinates. 
In particular, 
any finitely generated subgroup of $[[\phi]]$ has polynomial growth. 
\end{proposition}
\begin{proof}
Suppose that $(X,\phi)$ is the odometer of type $(m_n)_n$. 
For $k\in\N$ and $l\in\Z_{m_k}$, 
we set $U(k,l)=\{(x_n)_n\in X\mid x_k=l\}$. 
Then $\{U(k,l)\mid l\in\Z_{m_k}\}$ is a clopen partition of $X$. 
For $\psi\in[[\phi]]$, 
we let $c_\psi:X\to\Z$ be the continuous function 
satisfying $\psi(x)=\phi^{c_\psi(x)}(x)$. 
Define a subgroup $\Gamma_k\subset[[\phi]]$ by 
\[
\Gamma_k=\{\psi\in[[\phi]]
\mid\text{$c_\psi$ is constant on $U(k,l)$ for each $l\in\Z_{m_k}$}\}. 
\]
Clearly $\Gamma_k\subset\Gamma_{k+1}$ and 
the union of $\Gamma_k$ equals $[[\phi]]$. 

Fix $k\in\N$. 
We would like to show that 
$\Gamma_k$ is isomorphic to $\Z^{m_k}\rtimes S_{m_k}$. 
For any $\psi\in\Gamma_k$, 
there exists $\tau\in S_{m_k}$ such that $\psi(U(k,l))=U(k,\tau(l))$, 
and so we obtain a homomorphism $\pi:\Gamma_k\to S_{m_k}$. 
For each $\tau\in S_{m_k}$, one can define $\psi\in\Gamma_k$ by 
$\psi(x)=\phi^{\tau(l)-l}(x)$ for $x\in U(k,l)$, 
where $l$ and $\tau(l)$ are regarded as elements in $\{0,1,\dots,m_k{-}1\}$. 
The map $\tau\mapsto\psi$ is a homomorphism from $S_{m_k}$ to $\Gamma_k$ 
and is a right inverse of $\pi$. 
If $\psi\in\Gamma_k$ belongs to the kernel of $\pi$, 
then there exist $n_l\in\Z$ for $l\in\Z_{m_k}$ such that 
$\psi(x)=\phi^{n_lm_k}(x)$ holds for any $x\in U(k,l)$. 
Evidently, $\psi\mapsto(n_l)_l$ gives an isomorphism 
from $\Ker\pi$ to $\Z^{m_k}$. 
Consequently, $\Gamma_k$ is isomorphic to $\Z^{m_k}\rtimes S_{m_k}$. 
\end{proof}

The following lemma is well known. 
For the convenience of the reader, we include an explicit proof. 

\begin{lemma}
Let $(X,\phi)$ be a Cantor minimal system. 
If $(X,\phi)$ is not an odometer, 
then there exists a continuous map $\pi:X\to\{0,1\}^\Z$ 
such that $\pi(X)$ is infinite and $\pi\circ\phi=\sigma\circ\pi$, 
where $\sigma:\{0,1\}^\Z\to\{0,1\}^\Z$ is the shift. 
\end{lemma}
\begin{proof}
Let $\{O_n\mid n\in\N\}$ be the set of all clopen subsets of $X$. 
For each $n\in\N$, we define a continuous map $\pi_n:X\to\{0,1\}^\Z$ by 
\[
\pi_n(x)_k=\begin{cases}1&\phi^k(x)\in O_n\\0&\phi^k(x)\notin O_n\end{cases}
\]
for $x\in X$. 
Then one has $\pi_n\circ\phi=\sigma\circ\pi_n$. 

Arguing by contradiction, 
we assume that $\pi_n(X)$ is finite for all $n\in\N$. 
We will construct continuous maps $\tilde\pi_n:X\to\Z_{m_n}$ and 
surjective homomorphisms $\rho_n:\Z_{m_{n+1}}\to\Z_{m_n}$ such that 
$\tilde\pi_n(\phi(x))=\tilde\pi_n(x)+1$ for $x\in X$, 
$\tilde\pi_n=\rho_n\circ\tilde\pi_{n+1}$ and 
$\pi_n$ factors through $\tilde\pi_n$. 
First, letting $m_1=\lvert\pi_1(X)\rvert$, 
we can find a continuous map $\tilde\pi_1:X\to\Z_{m_1}$ 
such that $\tilde\pi_1(\phi(x))=\tilde\pi_1(x)+1$ and 
$\pi_1$ factors through $\tilde\pi_1$. 
Suppose that we have constructed $\tilde\pi_n:X\to\Z_{m_n}$. 
Consider the continuous map 
$\pi_{n+1}\times\tilde\pi_n:X\to\{0,1\}^\Z\times\Z_{m_n}$ and 
let $m_{n+1}=\lvert(\pi_{n+1}\times\tilde\pi_n)(X)\rvert$. 
Then, identifying $(\pi_{n+1}\times\tilde\pi_n)(X)$ with $\Z_{m_{n+1}}$, 
we can construct $\tilde\pi_{n+1}:X\to\Z_{m_{n+1}}$ and 
$\rho_n:\Z_{m_{n+1}}\to\Z_{m_n}$ as desired. 

Let $(Y,\psi)$ be the odometer of type $(m_n)_n$. 
Define the continuous map $f:X\to Y$ by $f(x)=(\tilde\pi_n(x))_n$. 
Clearly we have $f\circ\phi=\psi\circ f$. 
For any distinct points $x,x'\in X$, 
there exists $O_n$ such that $x\in O_n$ and $x'\notin O_n$, 
which means that $f$ is injective. 
Thus $(X,\phi)$ is conjugate to $(Y,\psi)$, 
which completes the proof. 
\end{proof}

In what follows, for a clopen subset $O\subset X$, 
\[
1_O:X\to\Z_2
\]
denotes the $\Z_2$-valued characteristic function of $O$. 
The following proposition is used in the proof of Theorem \ref{lamplighter} 
in order to construct an embedding 
of the infinite direct sum of $\Z_2$ into $[[\phi]]$. 

\begin{proposition}
Let $(X,\phi)$ be a Cantor minimal system. 
If $(X,\phi)$ is not an odometer, 
then there exists a clopen subset $O\subset X$ such that 
for any finite subset $F\subset\Z$, the function 
\[
\sum_{k\in F}1_O\circ\phi^k
\]
is not identically zero mod $2$. 
\end{proposition}
\begin{proof}
By the lemma above, 
there exists a continuous map $\pi:X\to\{0,1\}^\Z$ 
such that $\pi(X)$ is infinite and $\pi\circ\phi=\sigma\circ\pi$, 
where $\sigma:\{0,1\}^\Z\to\{0,1\}^\Z$ is the shift. 
We identify $\{0,1\}$ with $\Z/2\Z$. 
Set 
\[
O=\{x\in X\mid\pi(x)_0=1\}. 
\]
In other words, $1_O(x)=\pi(x)_0\in\Z/2\Z$. 
We would like to see that 
the clopen subset $O\subset X$ satisfies the requirement. 
Let $F\subset\Z$ be a finite subset. 
Suppose that the function 
\[
\sum_{k\in F}1_O\circ\phi^k
\]
is identically zero mod $2$. 
Put $l=\min\{k\in F\}$ and $m=\max\{k\in F\}$. 
Assume that 
$x,y\in X$ satisfies $\pi(x)_n=\pi(y)_n$ for all $n\in\{l,l{+}1,\dots,m\}$. 
Then 
\begin{align*}
\pi(x)_{l-1}
&=1_O(\phi^{l-1}(x))\\
&=1_O(\phi^{l-1}(x))+\sum_{k\in F}1_O(\phi^{k-1}(x))\\
&=\sum_{k\in F\setminus\{l\}}1_O(\phi^{k-1}(x))\\
&=\sum_{k\in F\setminus\{l\}}\pi(x)_{k-1}, 
\end{align*}
and so $\pi(x)_{l-1}=\pi(y)_{l-1}$. 
Repeating this procedure, we obtain $\pi(x)_n=\pi(y)_n$ for every $n\leq l$. 
In the same way we get $\pi(x)_n=\pi(y)_n$ for every $n\geq m$. 
Thus $\pi(x)=\pi(y)$. 
This means that the cardinality of $\pi(X)$ is at most $2^{m-l+1}$, 
which is a contradiction. 
\end{proof}

We call the wreath product 
\[
L=\left(\bigoplus_\Z\Z_2\right)\rtimes\Z
\]
the lamplighter group, 
where the semi-direct product is taken with respect to the shift action. 
It is easy to see that $L$ is finitely generated and 
that $L$ contains a free semi-group on two generators. 
Hence $L$ has exponential growth (see \cite[VII.1]{Harpe} for example). 

The technique we employ in the proof of the following theorem 
is essentially the same as that of \cite[Theorem 8.1]{DFG}. 
I am grateful to Koji Fujiwara for explaining this technique. 

\begin{theorem}\label{lamplighter}
Let $(X,\phi)$ be a Cantor minimal system. 
The following three conditions are equivalent. 
\begin{enumerate}
\item $(X,\phi)$ is not an odometer. 
\item $D([[\phi]])$ contains the lamplighter group $L$. 
\item $[[\phi]]$ contains the lamplighter group $L$. 
\end{enumerate}
\end{theorem}
\begin{proof}
(2)$\Rightarrow$(3) is obvious. 
(3)$\Rightarrow$(1) immediately follows from Proposition \ref{odmtr}. 
We show (1)$\Rightarrow$(2). 
Choose a non-empty clopen subset $U\subset X$ so that 
$U$, $\phi(U)$, $\phi^2(U)$ and $\phi^3(U)$ are disjoint. 
Let $\psi$ be the first return map on $U$ 
(see \cite[Definition 1.5]{GPS95crelle}). 
Letting $\psi(x)=x$ for $x\in X\setminus U$, 
we may regard $\psi$ as an element of $[[\phi]]$. 
Define $r=\psi\circ\phi\circ\psi\circ\phi^{-1}$. 
Clearly $r$ is of infinite order. 
For each clopen subset $V\subset U$, 
we define $\tau_V\in[[\phi]]$ by 
\[
\tau_V(x)=\begin{cases}\phi(x)&x\in V\\\phi^{-1}(x)&x\in\phi(V)\\
x&\text{otherwise.}\end{cases}
\]
It is easy to see that for any clopen subsets $V,W\subset U$, 
$\tau_V$ and $\tau_W$ commute. 
Also, we have $r\circ\tau_V\circ r^{-1}=\tau_{\psi(V)}$. 
Furthermore, for clopen subsets $V_1,V_2,\dots,V_n\subset U$, 
$\tau_{V_1}\circ\tau_{V_2}\circ\dots\circ\tau_{V_n}$ equals the identity 
if and only if $1_{V_1}+1_{V_2}+\dots+1_{V_n}$ equals zero 
(as a $\Z_2$-valued function). 
Since $(X,\phi)$ is not an odometer, 
neither is $(U,\psi|U)$. 
It follows from the proposition above that 
there exists a clopen subset $O\subset U$ such that 
for any finite subset $F\subset\Z$, the function 
\[
\sum_{k\in F}1_O\circ\psi^k
\]
is not identically zero mod $2$. 
Define $s=\tau_O$. 
Because $r^k\circ s\circ r^{-k}=\tau_{\psi^k(O)}$ for any $k\in\Z$, 
the homeomorphisms $r^k\circ s\circ r^{-k}$ commute with each other. 
Moreover, for any non-empty finite subset $\{k_1,k_2,\dots,k_n\}\subset\Z$, 
\[
(r^{k_1}\circ s\circ r^{-k_1})\circ(r^{k_2}\circ s\circ r^{-k_2})\circ
\dots\circ(r^{k_n}\circ s\circ r^{-k_n})
\]
is not the identity, because 
\[
1_{\psi^{k_1}(O)}+1_{\psi^{k_2}(O)}+\dots+1_{\psi^{k_n}(O)}\neq0. 
\]
Therefore, 
the subgroup generated by $r$ and $s$ is isomorphic to the lamplighter group. 
The support of $r$ and $s$ is contained in $U\cup\phi(U)$, and so 
the support of $\phi^2\circ r\circ\phi^{-2}$ and 
$\phi^2\circ s\circ\phi^{-2}$ is contained in $\phi^2(U)\cup\phi^3(U)$. 
Since $U\cup\phi(U)$ and $\phi^2(U)\cup\phi^3(U)$ are disjoint, 
the subgroup 
\[
\langle r\circ\phi^2\circ r^{-1}\circ\phi^{-2},\ 
s\circ\phi^2\circ s^{-1}\circ\phi^{-2}\rangle\subset D([[\phi]])
\]
is also isomorphic to the lamplighter group, 
which completes the proof. 
\end{proof}

\begin{corollary}\label{exponential}
Let $(X,\phi)$ be a Cantor minimal system. 
If $D([[\phi]])$ is finitely generated, then 
it has exponential growth. 
\end{corollary}
\begin{proof}
By Proposition \ref{odmtr}, $\phi$ is not an odometer. 
(We remark that $D([[\phi]])$ is finitely generated 
if and only if $\phi$ is a minimal subshift over a finite alphabet, 
see \cite[Theorem 5.4]{M06IJM}.) 
It follows from the theorem above that 
$D([[\phi]])$ contains the lamplighter group $L$. 
As mentioned above, $L$ has exponential growth. 
Therefore $D([[\phi]])$ has exponential growth, too. 
\end{proof}

\section{Generators of full groups}

In this section, we will prove that 
the measurable full group associated to the countable, measure-preserving, 
ergodic and hyperfinite equivalence relation is 
topologically generated by two elements (Theorem \ref{t=2}).

\subsection{Algebraic generators of topological full groups}

Let $\alpha\in(0,1)$ be an irrational number and 
let $(X,\phi)$ be the Sturmian shift 
arising from the $\alpha$-rotation on $\T=\R/\Z$. 
In \cite[Example 6.2]{M06IJM}, it was shown that 
the topological full group $[[\phi]]$ is (algebraically) generated 
by the three elements $\sigma_U$, $\sigma_V$ and $\phi$. 
In this subsection, for $0<\alpha<1/6$, we will show that 
the subgroup generated by $\phi$ and $\sigma_U$ contains 
the commutator subgroup $D([[\phi]])$. 

Assume $0<\alpha<1/6$. 
We recall the notation used in \cite[Example 6.2]{M06IJM}. 
The clopen subset corresponding to the interval $[0,\alpha)\subset\T$ 
is denoted by $U\subset X$. 
The element $\sigma_U\in[[\phi]]$ is defined by 
\[
\sigma_U(x)=\begin{cases}\phi(x)&x\in U\\\phi^{-1}(U)&x\in\phi(U)\\
x&\text{otherwise}. \end{cases}
\]

\begin{proposition}\label{sturmian}
In the setting above, let $G\subset[[\phi]]$ be the subgroup 
generated by $\phi$ and $\sigma_U$. 
\begin{enumerate}
\item The commutator subgroup $D([[\phi]])$ is contained in $G$. 
\item The subgroup $G$ is normal and $[[\phi]]/G\cong\Z/2\Z$. 
\end{enumerate}
\end{proposition}
\begin{proof}
(1)
Let $\mathcal{W}$ be the set of all clopen subsets $W\subset X$ 
such that $\phi^{-1}(W)$, $W$, $\phi(W)$ are mutually disjoint. 
Clearly $U$ is in $\mathcal{W}$. 
For $W\in\mathcal{W}$, we define $\gamma_W\in[[\phi]]$ by 
\[
\gamma_W(x)=\begin{cases}\phi(x)&x\in\phi^{-1}(W)\cup W\\
\phi^{-2}(x)&x\in\phi(W)\\x&\text{otherwise}. \end{cases}
\]
It is easy to see that $\gamma_W$ belongs to $D([[\phi]])$ 
(see the comment before \cite[Lemma 5.2]{M06IJM}). 
It was shown in \cite[Lemma 5.2]{M06IJM} that 
the commutator subgroup $D([[\phi]])$ is generated by 
$\{\gamma_W\mid W\in\mathcal{W}\}$ 
(see also the comment before \cite[Theorem 4.9]{M06IJM}). 
For any $n\in\Z$, $\phi^n(U)$ is the clopen subset 
corresponding to the interval $[n\alpha,(n{+}1)\alpha)\subset\T$. 
For any $m,n\in\Z$, if $\phi^m(U)\cap\phi^n(U)$ is not empty, then 
it corresponds to 
either $[m\alpha,(n{+}1)\alpha)$ or $[n\alpha,(m{+}1)\alpha)$. 
Hence any clopen subset $W\subset X$ can be written 
as a finite disjoint union of 
clopen subsets of the form $\phi^m(U)\cap\phi^n(U)$. 
It follows that $D([[\phi]])$ is generated by 
\[
\{\gamma_W\mid\exists m,n\in\Z,\ W{=}\phi^m(U)\cap\phi^n(U)\}. 
\]
Since $\phi^n\circ\sigma_U\circ\phi^{-n}=\sigma_{\phi^n(U)}$, one verifies 
\begin{align*}
&(\phi^n\circ\sigma_U\circ\phi^{-n})
\circ(\phi^{n-1}\circ\sigma_U\circ\phi^{-n+1})
\circ(\phi^n\circ\sigma_U\circ\phi^{-n})
\circ(\phi^{n-1}\circ\sigma_U\circ\phi^{-n+1})\\
&=\sigma_{\phi^n(U)}\circ\sigma_{\phi^{n-1}(U)}
\circ\sigma_{\phi^n(U)}\circ\sigma_{\phi^{n-1}(U)}\\
&=\gamma_{\phi^n(U)}, 
\end{align*}
and so $\gamma_{\phi^n(U)}$ belongs to $G$. 
Suppose that 
the clopen subset $\phi^m(U)\cap\phi^n(U)$ corresponds to 
the interval $[m\alpha,(n{+}1)\alpha)\subset\T$. 
Then $\phi^{n-2}(U)$, $\phi^{n-1}(U)$, $\phi^n(U)\cup\phi^m(U)$, 
$\phi^{m+1}(U)$, $\phi^{m+2}(U)$ are mutually disjoint, 
because $\alpha$ is less than $1/6$. 
Therefore, thanks to \cite[Lemma 5.3 (ii)]{M06IJM}, we have 
\[
\gamma_{\phi^{m+1}(U)}\circ\gamma_{\phi^{n-1}(U)}^{-1}
\circ\gamma_{\phi^{m+1}(U)}^{-1}\circ\gamma_{\phi^{n-1}(U)}
=\gamma_{\phi^m(U)\cap\phi^n(U)}, 
\]
and so $\gamma_{\phi^m(U)\cap\phi^n(U)}$ belongs to $G$. 
When the clopen subset $\phi^m(U)\cap\phi^n(U)$ corresponds to 
the interval $[n\alpha,(m{+}1)\alpha)\subset\T$, 
we obtain the same conclusion in a similar way. 
Hence $G$ contains the commutator subgroup $D([[\phi]])$. 

(2)
By \cite[Example 5.2]{M06IJM}, 
$[[\phi]]$ is generated by $\phi$, $\sigma_U$ and $\sigma_V$, 
where $V\subset X$ is another clopen subset. 
The element $\sigma_V$ is of order two and is not contained in $G$. 
We can check 
\[
\sigma_V\circ\phi\circ\sigma_V
=(\sigma_V\circ\phi\circ\sigma_V\circ\phi^{-1})\circ\phi\in G
\]
and 
\[
\sigma_V\circ\sigma_U\circ\sigma_V
=(\sigma_V\circ\sigma_U\circ\sigma_V\circ\sigma_U)\circ\sigma_U\in G. 
\]
It follows that $G$ is normal and $[[\phi]]/G\cong\Z/2\Z$. 
\end{proof}

\subsection{Topological generators of measurable full groups}

In this subsection we follow the notation of \cite{KT10Ergodic}. 
Let $X$ be a standard Borel space and 
let $\mu$ be a non-atomic Borel probability measure on it. 
Denote by $\Aut(X,\mu)$ the group of 
all measure-preserving automorphisms of $(X,\mu)$ (modulo null sets). 
We equip the group $\Aut(X,\mu)$ with the topology induced by the metric 
\[
d(f,g)=\mu(\{x\in X\mid f(x)=g(x)\})
\]
and call it the uniform topology. 
For a countable, Borel, measure-preserving equivalence relation $E$ on $X$, 
its measurable full group $[E]$ is defined by 
\[
[E]=\{f\in\Aut(X,\mu)\mid(x,f(x))\in E\text{ for almost every }x\in X\}. 
\]
The measurable full group $[E]$ is a closed subgroup of $\Aut(X,\mu)$ 
in the uniform topology and 
they turn out to be separable, and hence Polish. 
Following \cite{KT10Ergodic}, 
we let $t([E])$ denote the minimum number of topological generators of $[E]$ 
(i.e. the minimum number of elements 
which generate a dense subgroup of $[E]$). 
By using Proposition \ref{sturmian}, 
we can improve some results for $t([E])$ obtained in \cite{KT10Ergodic}. 

The following theorem answers \cite[Question 4.3]{KT10Ergodic}. 

\begin{theorem}\label{t=2}
Let $E$ be the countable, measure-preserving, ergodic and hyperfinite 
equivalence relation on the standard probability space $(X,\mu)$. 
Then $t([E])=2$. 
\end{theorem}
\begin{proof}
Let $(X,\phi)$ be the Sturmian shift 
arising from an irrational number $\alpha\in(0,1/6)$. 
By Proposition \ref{sturmian}, 
there exists $\sigma_U\in[[\phi]]$ such that 
the subgroup $G\subset[[\phi]]$ generated by $\phi$ and $\sigma_U$ 
contains the commutator subgroup $D([[\phi]])$. 

Let $E\subset X\times X$ be the equivalence relation induced by $\phi$. 
There exists a unique $\phi$-invariant Borel probability measure $\mu$ on $X$. 
Then $E$ is the countable, measure-preserving, ergodic and hyperfinite 
equivalence relation on the standard probability space $(X,\mu)$. 
By \cite[Proposition 4.1]{KT10Ergodic}, 
$[[\phi]]$ is dense in $[E]$ in the uniform topology. 
In particular, the commutator subgroup $D([[\phi]])$ is dense 
in the commutator subgroup $D([E])$. 
It follows that $G$ is dense in $D([E])$. 
Since $[E]$ is simple by \cite{E82Israel}, $[E]=D([E])$. 
Therefore the group $G$ generated by the two elements $\phi$ and $\sigma_U$ 
is dense in $[E]$, 
which implies $t([E])=2$. 
\end{proof}

The theorem above enables us to sharpen the estimates 
given in Theorem 4.10 and Corollary 4.12 of \cite{KT10Ergodic}. 

\begin{corollary}\label{cost}
Let $E$ be a countable, measure-preserving, ergodic equivalence relation 
on the standard probability space $(X,\mu)$. 
\begin{enumerate}
\item If the cost of $E$ is less than $n$ for some $n\in\N$, 
then $t([E])\leq2n$. 
\item If $E$ is induced by a free action of the free group $\mathbb{F}_n$, 
then $n+1\leq t([E])\leq2(n+1)$. 
\end{enumerate}
\end{corollary}

\bigskip
\bigskip

\noindent
\textbf{Acknowledgement. }
The author would like to thank Koji Fujiwara and Konstantin Medynets 
for valuable discussions.


\begin{thebibliography}{99}
\bibitem{BM}
S. Bezuglyi and K. Medynets, 
\textit{Full groups, flip conjugacy, and 
orbit equivalence of Cantor minimal systems}, 
Colloq. Math. 110 (2008), 409--429. 
arXiv:math/0611173
\bibitem{DFG}
F. Dahmani, K. Fujiwara and V. Guirardel, 
\textit{Free groups of interval exchange transformations are rare}, 
preprint. 
arXiv:1101.5909
\bibitem{E82Israel}
S. J. Eigen, 
\textit{On the simplicity of the full group of ergodic transformations}, 
Israel J. Math. 40 (1981), 345--349. 
\bibitem{GMPS}
T. Giordano, H. Matui, I. F. Putnam and C. F. Skau, 
\textit{Orbit equivalence for Cantor minimal $\Z^d$-systems}, 
Invent. Math. 179 (2010), 119--158. 
arXiv:0810.3957
\bibitem{GPS95crelle}
T. Giordano, I. F. Putnam and C. F. Skau, 
\textit{Topological orbit equivalence and $C^*$-crossed products}, 
J. Reine  Angew. Math. 469 (1995), 51--111. 
\bibitem{GPS99Israel}
T. Giordano, I. F. Putnam and C. F. Skau, 
\textit{Full groups of Cantor minimal systems}, 
Israel J. Math. 111 (1999), 285--320. 
\bibitem{GM}
R. Grigorchuk and K. Medynets, 
\textit{Topological full groups are locally embeddable into finite groups}, 
preprint. 
arXiv:1105.0719
\bibitem{Harpe}
P. de la Harpe, 
\textit{Topics in geometric group theory}, 
Chicago Lectures in Mathematics, 
University of Chicago Press, Chicago, IL, 2000. 
\bibitem{KT10Ergodic}
J. Kittrell and T. Tsankov, 
\textit{Topological properties of full groups}, 
Ergodic Theory Dynam. Systems 30 (2010), 525--545. 
\bibitem{M06IJM}
H. Matui, 
\textit{Some remarks on topological full groups of Cantor minimal systems}, 
Internat. J. Math. 17 (2006), 231--251. 
math.DS/0404117
\end{thebibliography}
\end{document}